\documentclass[12pt]{amsart}

\usepackage{amsmath}
\usepackage{amsthm}
\usepackage{amssymb}
\usepackage{caption}
\usepackage[curve]{xypic}
\usepackage{cite}
\usepackage{enumitem}
\usepackage{color}
\usepackage{url}
\usepackage[usenames,dvipsnames]{xcolor}
\usepackage{graphicx}
\usepackage{tikz}
\usepackage{tikz-cd}
\usepackage{hyperref} 
\usepackage{verbatim}
\usepackage{multirow}

\usepackage{algorithm}
\usepackage{algorithmic}

\setlist{itemsep=4pt, topsep=4pt}

\makeatletter

\def\chaptermark#1{}

\def\chapter{%
\if@openright\cleardoublepage\else\clearpage\fi
\thispagestyle{plain}\global\@topnum\z@
\@afterindenttrue \secdef\@chapter\@schapter}

\def\@chapter[#1]#2{\refstepcounter{chapter}%
\ifnum\c@secnumdepth<\z@ \let\@secnumber\@empty
\else \let\@secnumber\thechapter \fi
\typeout{\chaptername\space\@secnumber}%
\def\@toclevel{0}%
\ifx\chaptername\appendixname \@tocwriteb\tocappendix{chapter}{#2}%
\else \@tocwriteb\tocchapter{chapter}{#2}\fi
\chaptermark{#1}%
\addtocontents{lof}{\protect\addvspace{10\p@}}%
\addtocontents{lot}{\protect\addvspace{10\p@}}%
\@makechapterhead{#2}\@afterheading}
\def\@schapter#1{\typeout{#1}%
\let\@secnumber\@empty
\def\@toclevel{0}%
\ifx\chaptername\appendixname \@tocwriteb\tocappendix{chapter}{#1}%
\else \@tocwriteb\tocchapter{chapter}{#1}\fi
\chaptermark{#1}%
\addtocontents{lof}{\protect\addvspace{10\p@}}%
\addtocontents{lot}{\protect\addvspace{10\p@}}%
\@makeschapterhead{#1}\@afterheading}
\newcommand\chaptername{Chapter}

\def\@makechapterhead#1{\global\topskip 7.5pc\relax
\begingroup
\fontsize{\@xivpt}{18}\bfseries\centering
\ifnum\c@secnumdepth>\m@ne
  \leavevmode \hskip-\leftskip
  \rlap{\vbox to\z@{\vss
      \centerline{\normalsize\mdseries
          \uppercase\@xp{\chaptername}\enspace\thechapter}
      \vskip 3pc}}\hskip\leftskip\fi
 #1\par \endgroup
\skip@34\p@ \advance\skip@-\normalbaselineskip
\vskip\skip@ }
\def\@makeschapterhead#1{\global\topskip 7.5pc\relax
\begingroup
\fontsize{\@xivpt}{18}\bfseries\centering
#1\par \endgroup
\skip@34\p@ \advance\skip@-\normalbaselineskip
\vskip\skip@ }
\def\appendix{\par
\c@chapter\z@ \c@section\z@
\let\chaptername\appendixname
\def\thechapter{\@Alph\c@chapter}}

\newcounter{chapter}

\newif\if@openright

\makeatother

\makeatletter 
\def\@cite#1#2{{\m@th\upshape\bfseries%
[{#1\if@tempswa{\m@th\upshape\mdseries, #2}\fi}]}}
\makeatother 

\theoremstyle{plain}
\newtheorem{thm}{Theorem}

\newtheorem{prop}[thm]{Proposition}
\newtheorem{lem}[thm]{Lemma}
\newtheorem{sublem}[thm]{Sublemma}
\theoremstyle{definition}

\theoremstyle{remark}

\numberwithin{equation}{subsection}

\newcommand{\nc}{\newcommand}
\newcommand{\rnc}{\renewcommand}

\nc\bA{\mathbb{A}}
\nc\bB{\mathbb{B}}
\nc\bC{\mathbb{C}}
\nc\bD{\mathbb{D}}
\nc\bE{\mathbb{E}}
\nc\bF{\mathbb{F}}
\nc\bG{\mathbb{G}}
\nc\bH{\mathbb{H}}
\nc\bI{\mathbb{I}}
\nc{\bJ}{\mathbb{J}} 
\nc\bK{\mathbb{K}}
\nc\bL{\mathbb{L}}
\nc\bM{\mathbb{M}}
\nc\bN{\mathbb{N}}
\nc\bO{\mathbb{O}}
\nc\bP{\mathbb{P}}
\nc\bQ{\mathbb{Q}}
\nc\bR{\mathbb{R}}
\nc\bS{\mathbb{S}}
\nc\bT{\mathbb{T}}
\nc\bU{\mathbb{U}}
\nc\bV{\mathbb{V}}
\nc\bW{\mathbb{W}}
\nc\bY{\mathbb{Y}}
\nc\bX{\mathbb{X}}
\nc\bZ{\mathbb{Z}}
\nc\cA{\mathcal{A}}
\nc\cB{\mathcal{B}}
\nc\cC{\mathcal{C}}
\rnc\cD{\mathcal{D}}
\nc\cE{\mathcal{E}}
\nc\cF{\mathcal{F}}
\nc\cG{\mathcal{G}}
\rnc\cH{\mathcal{H}}
\nc\cI{\mathcal{I}}
\nc{\cJ}{\mathcal{J}} 
\nc\cK{\mathcal{K}}
\rnc\cL{\mathcal{L}}
\nc\cM{\mathcal{M}}
\nc\cN{\mathcal{N}}
\nc\cO{\mathcal{O}}
\nc\cP{\mathcal{P}}
\nc\cQ{\mathcal{Q}}
\rnc\cR{\mathcal{R}}
\nc\cS{\mathcal{S}}
\nc\cT{\mathcal{T}}
\nc\cU{\mathcal{U}}
\nc\cV{\mathcal{V}}
\nc\cW{\mathcal{W}}
\nc\cY{\mathcal{Y}}
\nc\cX{\mathcal{X}}
\nc\cZ{\mathcal{Z}}
\nc\bfA{\mathbf{A}}
\nc\bfB{\mathbf{B}}
\nc\bfC{\mathbf{C}}
\nc\bfD{\mathbf{D}}
\nc\bfE{\mathbf{E}}
\nc\bfF{\mathbf{F}}
\nc\bfG{\mathbf{G}}
\nc\bfH{\mathbf{H}}
\nc\bfI{\mathbf{I}}
\nc{\bfJ}{\mathbf{J}} 
\nc\bfK{\mathbf{K}}
\nc\bfL{\mathbf{L}}
\nc\bfM{\mathbf{M}}
\nc\bfN{\mathbf{N}}
\nc\bfO{\mathbf{O}}
\nc\bfP{\mathbf{P}}
\nc\bfQ{\mathbf{Q}}
\nc\bfR{\mathbf{R}}
\nc\bfS{\mathbf{S}}
\nc\bfT{\mathbf{T}}
\nc\bfU{\mathbf{U}}
\nc\bfV{\mathbf{V}}
\nc\bfW{\mathbf{W}}
\nc\bfY{\mathbf{Y}}
\nc\bfX{\mathbf{X}}
\nc\bfZ{\mathbf{Z}}

\nc{\dmo}{\DeclareMathOperator}
\nc{\wt}{\widetilde}
\rnc{\Re}{\operatorname{Re}}
\rnc{\Im}{\operatorname{Im}}
\rnc{\span}{\operatorname{span}}
\dmo{\rank}{rank}
\dmo{\End}{End}
\dmo{\Hom}{Hom}
\dmo{\Jac}{Jac}
\dmo{\Id}{Id}
\dmo{\Ann}{Ann}
\dmo{\Area}{Area}
\dmo{\CP}{\bC P^1}
\dmo{\rk}{rk}
\dmo{\rel}{rel}
\dmo{\ra}{\rightarrow}
\dmo{\AGL}{\mathrm{AGL}}
\dmo{\AO}{\mathrm{AO}}
\dmo{\Sym}{\mathrm{Sym}}
\dmo{\Hur}{\mathrm{Hur}}
\dmo{\Aut}{\mathrm{Aut}}

\rnc{\Col}{\operatorname{Col}}

\nc{\ColOne}{\Col_{\bfC_1}}
\nc{\ColOneX}{\ColOne(X,\omega)}

\nc{\ColTwo}{\Col_{\bfC_2}}
\nc{\ColTwoX}{\ColTwo(X,\omega)}

\nc{\ColThree}{\Col_{\bfC_3}}
\nc{\ColThreeX}{\ColThree(X,\omega)}

\nc{\ColOneTwo}{\Col_{\bfC_1, \bfC_2}}
\nc{\ColOneTwoX}{\ColOneTwo(X,\omega)}

\nc{\ColOneThree}{\Col_{\bfC_1, \bfC_3}}
\nc{\ColOneThreeX}{\ColOneThree(X,\omega)}

\nc{\MOne}{\cM_{\bfC_1}}
\nc{\MTwo}{\cM_{\bfC_2}}
\nc{\MOneTwo}{\cM_{\bfC_1, \bfC_2}}

\nc{\MThree}{\cM_{\bfC_3}}

\nc{\MOneThree}{\cM_{\bfC_1, \bfC_3}}

\dmo{\For}{\cF}

\nc{\GL}{\mathrm{GL}^+(2, \bR)}

\renewcommand{\color}[1]{\unskip}

\title[Higher rank in genus three]{A short proof of the classification of higher rank invariant subvarieties in genus three}
\author[Apisa]{Paul~Apisa}

\begin{document}

\begin{abstract}
We give a new short proof of the classification of rank at least two invariant subvarieties in genus three, which is due to Aulicino, Nguyen, and Wright. The proof uses techniques developed in recent work of Apisa-Wright.
\end{abstract}

\maketitle


\vspace{-5mm}

\section{Introduction}

The moduli space $\Omega \cM_g$ of pairs $(X, \omega)$ where $X$ is a genus $g$ Riemann surface and $\omega$ is a holomorphic $1$-form on $X$ admits a stratification by prescribing the number of zeros of $\omega$ and their orders of vanishing. These strata are invariant under the $\mathrm{GL}(2, \bR)$ action on $\Omega \cM_g$ generated by complex scalar multiplication, which fixes $X$ and acts on $\omega$, and Teichm\"uller geodesic flow. By work of Eskin-Mirzakhani \cite{EM} and Eskin-Mirzakhani-Mohammadi \cite{EMM}, the $\mathrm{GL}(2, \bR)$ orbit closure of a point in a stratum is an orbifold that is defined by real homogeneous linear equations in a natural coordinate system, i.e. period coordinates. These linear orbifolds are called \emph{invariant subvarieties} (the fact that they are subvarieties is due to Filip \cite{Fi1}). Given $(X, \omega)$ with zeros $\Sigma$ in a stratum $\cH$, $T_{(X, \omega)} \cH \cong H^1(X, \Sigma; \bC)$. Moreover, if $(X, \omega)$ is contained in an invariant subvariety $\cM$, then the image of the projection $p: T_{(X, \omega)} \cM \ra H^1(X; \bC)$ is symplectic by work of Avila-Eskin-M\"oller \cite{AEM}. The \emph{rank} of $\cM$ is defined to be half of the dimension of the image. The \emph{rel} of $\cM$ is the dimension of $\ker(p)$. The object of this note is to offer a new short proof of the following result due to Aulicino, Nguyen, and Wright \cite{NW,ANW,AN,AN2}.
\begin{thm}\label{T:Main}
The rank at least two invariant subvarieties in genus three strata are full loci of branched covers.
\end{thm}

We note that McMullen \cite{Mc5} achieved a classification of invariant subvarieties in genus two with no rank restriction and Ygouf \cite{Ygouf-Rank1Genus3} classified rank one rel one invariant subvarieties in genus three strata with at most two zeros.

Fundamental to the proof of Theorem \ref{T:Main} will be the notion of a cylinder equivalence class, cylinder deformations, and the cylinder deformation theorem of Wright \cite{Wcyl} (see Apisa-Wright \cite{ApisaWrightDiamonds} for definitions and notation). 

The two types of invariant subvarieties that will be most useful in the sequel are \emph{geminal} and \emph{cylinder rigid} ones. A geminal invariant subvariety $\cM$ is one such that for each cylinder $C$ on $(X, \omega) \in \cM$, either $C$ is \emph{free}, i.e. all cylinder deformations in $C$ remain in $\cM$ or $C$ has a \emph{twin} cylinder $C'$, i.e. one of equal height and circumference, so that all cylinder deformations that deform $C$ and $C'$ equally remain in $\cM$. The main examples of geminal invariant subvarieties are Abelian and quadratic doubles, i.e. full loci of double covers of strata of Abelian and quadratic differentials. 


A \emph{cylinder rigid} invariant subvariety $\cM$ is one in which every equivalence class can be partitioned into subequivalence classes so that subequivalent cylinders remain subequivalent as long as they persist. A \emph{subequivalence class} of cylinders is a collection of cylinders that have a constant ratio of heights (belonging to a fixed finite set only depending on $\cM$) under all deformations that remain in $\cM$ and so that the standard shear in this collection of cylinders belongs to the tangent space. Examples of cylinder rigid invariant subvarieties include rel zero invariant subvarieties, see \cite[Lemma 4.2]{Apisa-MHD}, and invariant subvarieties in their boundaries, see \cite[Proposition 4.12]{Apisa-MHD}. 

We will be interested in cylinder rigid rank one rel one invariant subvarieties of tori. Let $\cN \subseteq \cH(0^m)$ be such a subvariety and let $(X, \omega)$ be a surface in $\cN$. Consider a cylinder direction, without loss of generality the horizontal, consisting of two subequivalence classes of cylinders $\bfA$ and $\bfB$ and shear $\bfA$ and $\bfB$ so that there is only one subequivalence class of cylinders in the vertical direction and so that any core curve of a vertical cylinder crosses the core curve of any horizontal cylinder exactly once. Given such a core curve, it travels distance $a_1$ in $\bfA$, then distance $b_1$ in $\bfB$, then distance $a_2$ in $\bfA$, etc. Call each component of $\overline{\bfA}$ or $\overline{\bfB}$ a \emph{block} and say that its \emph{height} is the distance between opposite boundary components. See the left side of Figure \ref{F:Overcollapse}.

\begin{lem}\label{L:CylinderRigidTori}
Let $\cN$ be a cylinder rigid rank one rel one invariant subvariety of tori. Let $(X, \omega) \in \cN$ be a horizontally periodic surface whose horizontal cylinders are partitioned into two subequivalence classes, $\bfA$ and $\bfB$. Then any two $\bfA$ (resp. $\bfB$) blocks have identical heights. Moreover, each $\bfA$ block consists of either one cylinder or two cylinders of equal height. 
\end{lem}
\begin{proof}
Recall that \emph{overcollapsing a subequivalence class $\bfC$} means applying a standard shear to $\bfC$ to shrink its height to zero, which amounts to traveling along a line segment in $\cN$, and then continuing slightly further along this line (see Apisa \cite[Lemma 4.4]{Apisa-Codim1Hyp} for a definition and a similar argument). An illustration of overcollapsing $\bfB$ is shown in Figure \ref{F:Overcollapse}.

\begin{figure}[h]
    \centering
    \includegraphics[width=.5\linewidth]{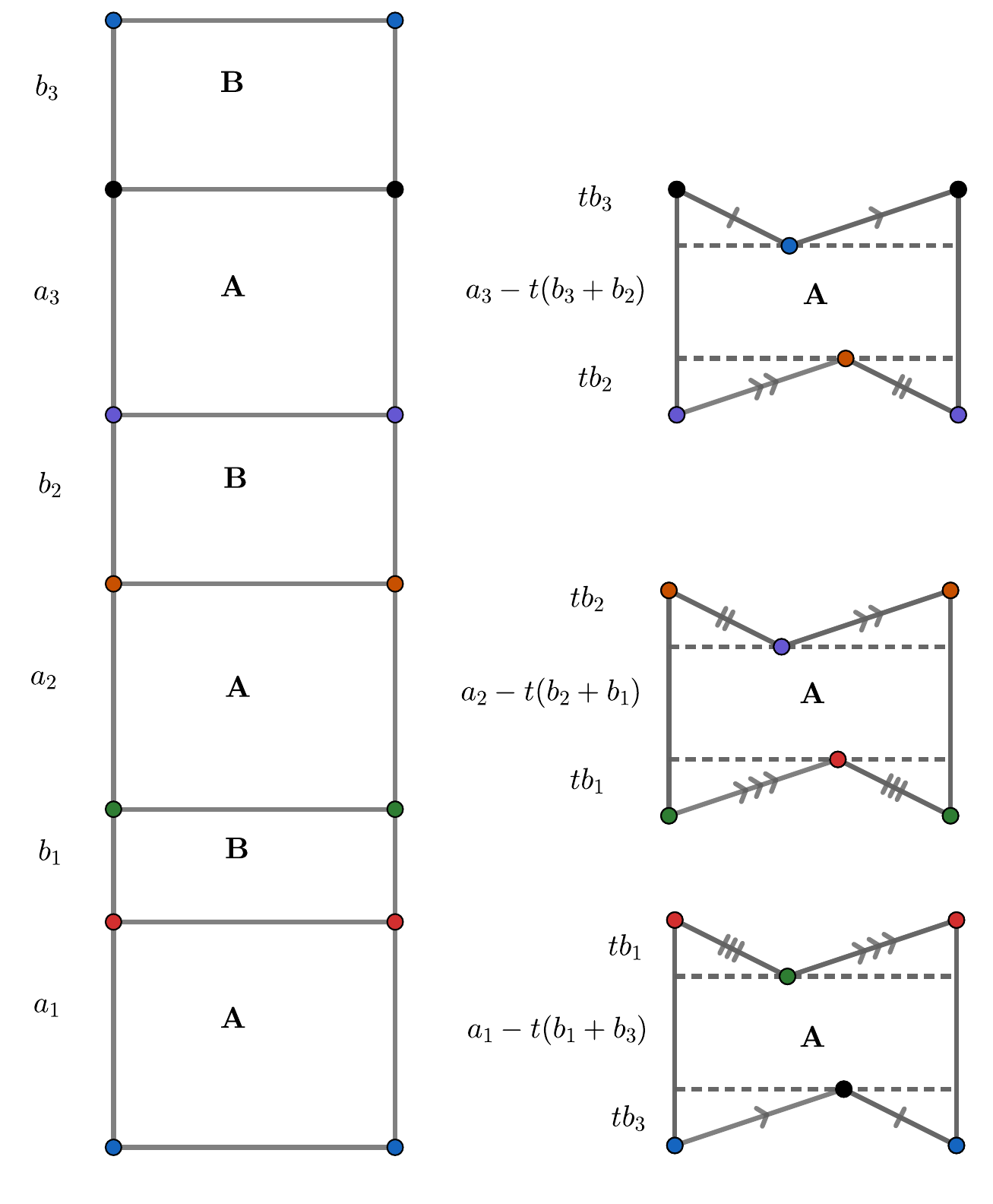}
    \caption{Overcollapsing $\bfB$. Opposite sides are identified unless otherwise indicated and $t$ is any sufficiently small positive real number. The figure on the left is the original surface $(X, \omega)$ and the figure on the right is the result of overcollapsing $\bfB$.}
    \label{F:Overcollapse}
\end{figure}

By definition of cylinder rigidity, the ratio of heights of cylinders (and hence of blocks) in $\bfA$ remain constant while overcollapsing. So, after applying a standard shear to $\bfA$ to make the first equality true, we see by overcollapsing $\bfA$ and $\bfB$ that there is some constant $c$ so that the following equations hold,
\[ a_1 + a_2 = b_1 \quad b_1 + b_2 = c \cdot a_2 \quad a_2 + a_3 = b_2 \quad b_2 + b_3 = c\cdot a_3 \quad etc. \]
Letting $A = \sum_i a_i$ and $B = \sum_i b_i$ we see that, provided that there are at least two $\bfA$ blocks, $2A = B$ and that $2B = cA$, which implies that $c = 4$. If $n$ is the number of $\bfA$ blocks, then this is a system of $2n$ equations in $2n$ unknowns that has rank at least $2n-1$ (since each subsequent equation involves a variable that has not appeared in a previous equation until we arrive at the equation $b_n + b_1 = a_1$). Since the system is homogeneous it has rank exactly equal to $2n-1$. Therefore, all solutions are multiples of each other, so the only solution has the form $a=a_i$ for all $i$, $b = b_i$ for all $i$, and $2a=b$ for some positive constant $a$.  Rephrased, each $\bfA$ (resp. $\bfB$) block has height $a$ (resp. $b$). By overcollapsing again, as in Apisa \cite[Lemma 4.4]{Apisa-MHD}, we see that each block consists of either a single cylinder or two cylinders of equal height.
\end{proof}

\begin{lem}\label{L:LoseTwoGenus}
Let $\cM$ be a rank two rel zero invariant subvariety in a genus three stratum without marked points. Suppose that $\cM$ is not a quadratic double. Let $\bfC$ be a generic equivalence class of cylinders on a surface $(X, \omega) \in \cM$ so that two core curves of cylinders in $\bfC$ are not homologous. Then every equivalence class in the complement of $\bfC$ consists of either one free cylinder or two cylinders with pairwise homologous core curves. Moreover, $\Col_{\bfC}(X, \omega)$ is genus one, $\bfC$ has exactly two cylinders, and, after possibly perturbing to a nearby surface in $\cM$, the complement of $\bfC$ contains an equivalence class with two cylinders whose core curves are homologous.
\end{lem}
See \cite[Definition 4.15]{ApisaWrightHighRank} for the definition of a \emph{generic equivalence class}. Recall that $\Col_\bfC(X, \omega)$ is the result of applying a cylinder deformation, specifically a standard dilation, to $\bfC$ to send the heights of the cylinders in it to zero while collapsing a saddle connection. This surface is contained in an invariant subvariety $\cM_{\bfC}$ in the boundary of $\cM$. $\cM_{\bfC}$ has rank one rel one by \cite[Lemma 6.5]{ApisaWright} and is cylinder rigid by \cite[Proposition 4.12]{Apisa-MHD}. So Lemma \ref{L:LoseTwoGenus} allows us to apply Lemma \ref{L:CylinderRigidTori} to $\cM_{\bfC}$.
\begin{proof}
Let $\bfD$ be an equivalence class in the complement of $\bfC$. We will first show that the cylinders in $\bfD$ must have pairwise homologous core curves. If not, then $\bfC \cup \bfD$ consists of four disjoint cylinders whose core curves satisfy some homology relation, which yields the contradiction that the cylinders in $\bfC$ and $\bfD$ are generically parallel.

Suppose now that no equivalence class in the complement of $\bfC$ contains two homologous cylinders, i.e. suppose that every cylinder in the complement of $\bfC$ is free, and that the same holds in some neighborhood of $(X, \omega)$ in $\cM$. By the prototype lemma \cite[Lemma 9.2]{Apisa}, we may perturb the surface in $\cM$ so that the complement of $\bfC$ contains one free cylinder nested in another. But this implies that $\cM$ is a quadratic double by the nested free cylinder theorem \cite[Theorem 7.2]{ApisaWrightHighRank}, which contradicts our hypotheses.

Therefore, after perhaps perturbing, we can suppose that the complement of $\bfC$ contains an equivalence class $\bfD$ composed of two cylinders with homologous core curves.  Each component of $(X, \omega) -\bfD$ is, topologically, a torus with two boundary components. Each of these components contains at most one cylinder in $\bfC$ since otherwise cutting the core curves of $\bfC \cup \bfD$ would produce a component that was a pair of pants with two boundary curves in $\bfC$ and one in $\bfD$, implying that the cylinders in $\bfD$ were generically parallel to those in $\bfC$, a contradiction. Therefore, each component of $(X, \omega) - \bfD$ contains exactly one cylinder from $\bfC$. So $\Col_{\bfC}(X, \omega)$ has genus one by Apisa-Wright \cite[Lemma 10.4]{ApisaWrightHighRank}.
%
%
\end{proof}

\begin{prop}\label{P:Base}
If $\cM$ is a rank two rel zero invariant subvariety in a genus three stratum then it is a full locus of branched covers.
\end{prop}
\begin{proof}
Suppose in order to deduce a contradiction that $\cM$ is not a full locus of branched covers and that the surfaces in it have no marked points. By Apisa \cite[Theorems 1.5 and 1.7]{Apisa-MHD}, since $\cM$ is not a locus of branched covers, there is a generic equivalence class of cylinders $\bfC_1$ on a surface $(X, \omega) \in \cM$ whose core curves span a subspace of absolute homology of dimension at least two. 


Let $\bfB$ be any equivalence class whose cylinders do not intersect those of $\bfC_1$. By Lemma \ref{L:LoseTwoGenus} we may suppose that $\bfB$ contains two cylinders with homologous core curves so that each component of $(X, \omega) - \bfB$ contains one cylinder from $\bfC_1$. By applying an element of $\mathrm{GL}(2, \bR)$, one may suppose that $\bfC_1$ is vertical and $\bfB$ is horizontal. Up to applying the standard shear to $\bfC_1$ and $\bfB$ one may suppose that $(X, \omega)$ is vertically and horizontally periodic with vertical cylinder equivalence classes $\bfC_1$ and $\bfC_2$ and horizontal cylinder equivalence classes $\bfA$ and $\bfB$ so that $\bfC_1$ is contained in $\bfA$ and $\bfB$ is contained in $\bfC_2$. (This fact was called the prototype lemma in \cite[Lemma 9.2]{Apisa}.) By Lemma \ref{L:LoseTwoGenus}, $\bfC_2$ contains either one free cylinder or two cylinders with homologous core curves and $\bfA$ contains two cylinders, one on each component of $(X, \omega) - \bfB$. We will deduce the contradiction that $\cM$ is an Abelian double of $\cH(2)$. 

Suppose, possibly after shearing $\bfB$, that the core curves of $\bfC_2$ pass through the cylinders in $\bfA$ and $\bfB$ exactly once. (Since $\Col_{\bfC_1}(X, \omega)$ is genus one, it is easy to see that this is possible to arrange after collapsing $\bfC_1$ and so it is possible beforehand too). Since $\bfC_2$ contains either one cylinder or two homologous ones, it follows from Lemma \ref{L:CylinderRigidTori}, applied to $\Col_{\bfC_1}(\bfA)$ and $\Col_{\bfC_1}(\bfB)$ on $\Col_{\bfC_1}(X, \omega)$, that each cylinder in $\bfA$ contains either one saddle connection or two of equal length from $\Col_{\bfC_1}(\bfC_1)$. 

\begin{figure}[h]
    \centering
    \includegraphics[width=.85\linewidth]{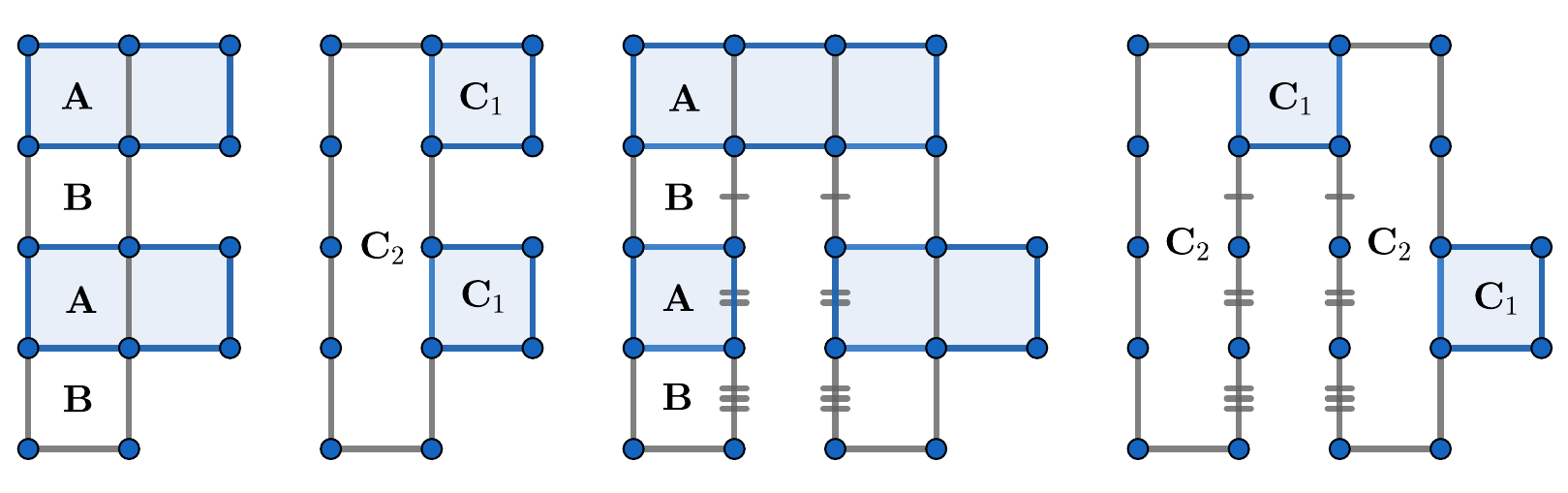}
    \caption{The two possible cases. Opposite sides are identified unless otherwise indicated. The case with $\bfC_2$ containing one (resp. two) cylinders appears on the left (resp. right). Each figure occurs twice to indicate the horizontal and vertical cylinders.}
    \label{F:Example}
\end{figure}


Since a combination of the standard shears of $\Col_{\bfB}(\bfC_1)$ and $\Col_{\bfB}(\bfC_2)$ must be rel on $\Col_{\bfB}(X, \omega)$, no cylinder in $\bfC_1$ has a saddle connection that appears in both of its boundaries. Since $(X, \omega)$ has no marked points and each component of $(X, \omega) - \bfB$ contains exactly one cylinder from $\bfC_1$, $\bfC_1$ is a pair of simple cylinders.

If there is just one cylinder in $\bfC_2$, we conclude that $(X, \omega)$ appears as the left two surfaces in Figure \ref{F:Example}. If there are two cylinders in $\bfC_2$, then the two cylinders in $\bfC_1$ belong to different components of $(X, \omega) - \bfC_2$, and we conclude that $(X, \omega)$ appears as the right two surfaces in Figure \ref{F:Example}. By Lemma \ref{L:CylinderRigidTori}, applied to $\Col_{\bfC_1}(X, \omega)$, the cylinders in $\bfA$ and $\bfB$ have the same heights. The same must also hold for $\bfC_1$ (for instance by the cylinder proportion theorem of Nguyen-Wright \cite[Proposition 3.2]{NW}) and also for $\bfC_2$ (for instance, when $\bfC_2$ has two cylinders, by symmetry of hypotheses on $\bfB$ and $\bfC_2$). This information completely describes the surface $(X, \omega)$, which we see is a double cover of surface in $\cH(2)$, a contradiction. 
\end{proof}

\begin{proof}[Proof of Theorem \ref{T:Main}:]
Suppose that $\cM$ is a rank at least two invariant subvariety without marked points in a genus three stratum. By Mirzakhani-Wright \cite{MirWri2}, if $\cM$ has rank three then it is a stratum or hyperelliptic locus. So suppose that the rank of $\cM$ is two. We will show that $\cM$ is a full locus of branched covers by induction on rel with Proposition \ref{P:Base} forming the base case. If $\cM$ is geminal, then it is an Abelian or quadratic double by Apisa-Wright \cite[Theorem 1.1]{ApisaWrightGeminal}. So suppose that $\cM$ is not geminal, i.e. suppose that there is a generic horizontal equivalence class $\bfC$ of cylinders on a surface $(X, \omega) \in \cM$ that certifies that $\cM$ is not geminal, i.e. $\bfC$ cannot be partitioned into twins and free cylinders.

\begin{sublem}
If $\bfD$ is a generic equivalence class of cylinders that is involved with rel on a surface $(X, \omega)$ in $\cM$ and $v$ is a typical degeneration in the twist space of $\bfD$, then $\cM_v$ is geminal.
\end{sublem}
The definition of a typical degeneration in the twist space and the boundary invariant subvariety $\cM_v$ can be found in \cite{ApisaWrightHighRank}.
\begin{proof}
The hypotheses imply that $\cM_v$ is prime by Apisa-Wright \cite[Lemma 9.1]{ApisaWrightHighRank} and rank two by Apisa-Wright \cite[Lemma 11.4]{ApisaWrightHighRank}. So $\Col_v(X, \omega)$ is connected. Since genus-preserving degenerations do not produce marked points, if $\Col_v(X, \omega)$ has genus three, then $\cM_v$ is an Abelian or quadratic double by the induction hypothesis. If $\Col_v(X, \omega)$ has genus two, then since $\cM_v$ is full rank and has no free marked points (by Apisa-Wright \cite[Lemma 11.6]{ApisaWrightHighRank}), it is a quadratic double in a genus two stratum.  
\end{proof}

The previous sublemma is enough to apply the arguments in Apisa-Wright \cite[Lemma 12.1 through Lemma 12.3]{ApisaWrightHighRank} to reduce to the case where all the rel of $\cM$ is supported on $\bfC$ and where there is a rel deformation supported on two homologous cylinders $C_1$ and $C_3$ in $\bfC$ of the form $w := i(\gamma_1^* - \gamma_3^*)$ where $\gamma_i$ is the core curve of $C_i$ and $\gamma_i^*$ is its Poincare dual. Moreover, $\bfC$ contains a third cylinder $C_2$ so that $\Col_w(C_2)$ and $\Col_w(C_3)$ are twins on $\Col_w(X, \omega)$ and similarly for $\Col_{-w}(C_2)$ and $\Col_{-w}(C_3)$ on $\Col_{-w}(X, \omega)$

\begin{sublem}
The cylinders in $\bfC$ are simple or half-simple and $\cM_{\pm w}$ is a quadratic double.
\end{sublem}
\begin{proof}
We begin by showing that any cylinder in the complement of $\bfC$ is free. Note that $(X, \omega) - \bigcup_{i=1}^3 \gamma_i$ is the disjoint union of a sphere with four boundary components and a torus with two boundary components. Only the latter can contain equivalence classes in the complement of $\bfC$ and so each such equivalence class must consist of a single free cylinder. 

As in the prototype lemma or \cite[Lemma 7.10 (6)]{ApisaWrightHighRank}, we may perturb $(X, \omega)$ in $\cM$ so that it is square-tiled, with the horizontal covered by $\bfC$ and another equivalence class $\bfC'$, and so that there is at least one vertical equivalence class $\bfD$ that does not intersect $\bfC$. Since $\bfC'$ contains a single free cylinder, call it $C$, $C$ must contain $\bfD$ and hence contains a nested simple cylinder. This nested free cylinder persists on $\cM_{\pm w}$, which is cylinder rigid by the induction hypothesis. By the nested free cylinder theorem \cite[Theorem 7.2]{ApisaWrightHighRank}, which we note applies to cylinder rigid invariant subvarieties, $\cM_{\pm w}$ is a quadratic double. 

By Masur-Zorich \cite{MZ}, see also \cite[Section 4.1]{ApisaWrightDiamonds}, there are five types of subequivalence classes of generic cylinders on a quadratic double. For three of them, those coming from complex cylinders, complex envelopes, and half-simple cylinders, all cylinders in at least one component of their complement come in twins. The remaining two subequivalence classes are preimages of simple cylinders and simple envelopes under the quotient by the holonomy involution. Since no cylinder in the complement of $\Col_{\pm w}(\bfC)$ has a twin, $\Col_{\pm w}(\bfC)$ and hence $\bfC$ consists of half-simple and simple cylinders.
\end{proof}

The previous sublemma is a replacement for \cite[Corollary 12.7]{ApisaWrightHighRank}. With this substitution, the proof follows verbatim the argument beginning after Corollary 12.7 in \cite{ApisaWrightHighRank}.
\end{proof}


\noindent \textbf{Acknowledgements.} The author was partially supported by NSF Grant DMS Award No. 2304840.

\bibliography{mybib}{}
\bibliographystyle{amsalpha}
\end{document}